\documentclass[12pt]{amsart}
\usepackage{amsfonts}
\usepackage{amsfonts,latexsym,rawfonts,amsmath,amssymb,amsthm}
\usepackage[plainpages=false]{hyperref}
\usepackage{graphicx}

\numberwithin{equation}{section}

\RequirePackage{color}
\theoremstyle{plain}

\newtheorem{theorem}{Theorem}[section]

\newtheorem{corollary}{Corollary}[section]
\newtheorem{definition}{Definition}[section]

\newtheorem{lemma}[theorem]{Lemma}
\newtheorem{proposition}[theorem]{Proposition}

\newtheorem{problem}[theorem]{Problem}

\newcommand{\beq}{\begin{equation}}
\newcommand{\eeq}{\end{equation}}
\newcommand{\beqs}{\begin{eqnarray*}}
\newcommand{\eeqs}{\end{eqnarray*}}
\newcommand{\beqn}{\begin{eqnarray}}
\newcommand{\eeqn}{\end{eqnarray}}
\newcommand{\beqa}{\begin{array}}
\newcommand{\eeqa}{\end{array}}

\def\phi{\varphi}

\begin{document}
\title[Prescribed curvature problem]{Convex hypersurfaces
of prescribed curvatures in hyperbolic space}

\author{Li Chen}
\address{Faculty of Mathematics and Statistics, Hubei Key Laboratory of Applied Mathematics, Hubei University,  Wuhan 430062, P.R. China}
\email{chenli@hubu.edu.cn}

\keywords{prescribed curvature problem; $h$-convex; Hessian quotient
equation.}

\subjclass[2010]{Primary 35J96, 52A39; Secondary 53A05.}


\begin{abstract}
For a smooth, closed and uniformly $h$-convex hypersurface $M$ in $\mathbb{H}^{n+1}$, the horospherical
Gauss map $G: M \rightarrow \mathbb{S}^n$ is a diffeomorphism.
We consider the problem of finding a smooth, closed and uniformly $h$-convex hypersurface $M\subset \mathbb{H}^{n+1}$
whose $k$-th shifted mean curvature $\widetilde{H}_{k}$ ($1\leq k\leq n$) is prescribed as a positive function $\tilde{f}(x)$ defined on $\mathbb{S}^n$, i.e.
\begin{eqnarray*}
\widetilde{H}_{k}(G^{-1}(x))=\tilde{f}(x).
\end{eqnarray*}
We can prove the existence of solution to this problem if the given function $\tilde{f}$ is even.
The similar problem has been considered by Guan-Guan
for convex hypersurfaces in Euclidean space two decades ago.
\end{abstract}

\maketitle

\baselineskip18pt

\parskip3pt

\section{Introduction}

For a smooth, closed and uniformly convex hypersurface $\Sigma$ in $\mathbb{R}^{n+1}$, the Gauss
map $n: \Sigma\rightarrow \mathbb{S}^n$ is a diffeomorphism. Then, the Weingarten matrix
\begin{eqnarray*}
\mathcal{W}=d n
\end{eqnarray*}
is positive definite. The $k$-th fundamental symmetric function
of the principal curvatures $\kappa_1, ..., \kappa_n$ (or the eigenvalues of $\mathcal{W}$)
\begin{eqnarray*}
H_k=\sum_{i_1<i_2< \cdot\cdot\cdot<
i_k}\kappa_{i_1}\cdot\cdot\cdot\kappa_{i_k}
\end{eqnarray*}
is called
the $k$-th mean curvature of $M$.
A fundamental question in classical
differential geometry concerns how much one can recover through the inverse
Gauss map when some information is prescribed on $\mathbb{S}^n$ \cite{Yau}. The most notable
example is probably the problem of finding
a closed and uniformly convex hypersurface in $\mathbb{R}^{n+1}$ whose
$k$-th mean curvature ($1\leq k\leq n$) is prescribed as a positive function $f$
defined on $\mathbb{S}^n$, i.e.
\begin{eqnarray}\label{GG}
H_{k}(n^{-1}(x))=f(x).
\end{eqnarray}
For $k=n$, this is the famous Minkowski problem (see \cite{Sch13} for a comprehensive introduction).
For $1\leq k<n$, this problem has been considered by Guan-Guan in \cite{Guan-g}.

In this paper, we want to consider the similar problem for uniformly $h$-convex hypersurfaces in hyperbolic space.
In order to formulate this problem, at first, we will briefly describe $h$-convex geometry in hyperbolic space
followed by Section 5 in \cite{And20} and Section 2 in \cite{Li-Xu}. Their works deeply declared
interesting formal similarities between the geometry of
$h$-convex domains in hyperbolic space and that of
convex Euclidean bodies.

We shall work
in the hyperboloid model of $\mathbb{H}^{n+1}$. For that, consider the
Minkowski space $\mathbb{R}^{n+1,1}$ with canonical coordinates
$X=(X^1, . . . , X^{n+1}, X^0)$ and the Lorentzian metric
\begin{eqnarray*}
\langle X, X\rangle=\sum_{i=1}^{n+1}(X^i)^2-(X^0)^2.
\end{eqnarray*}
$\mathbb{H}^{n+1}$ is the future time-like hyperboloid in Minkowski
space $\mathbb{R}^{n+1,1}$, i.e.
\begin{eqnarray*}
\mathbb{H}^{n+1}=\Big\{X=(X^1, \cdot\cdot\cdot, X^{n+1}, X^0) \in
\mathbb{R}^{n+1,1}: \langle X, X\rangle=-1, X^0>0\Big\}.
\end{eqnarray*}
The horospheres are hypersurfaces in $\mathbb{H}^{n+1}$ whose
principal curvatures equal to $1$ everywhere. In the hyperboloid model of $\mathbb{H}^{n+1}$,
they can be parameterized by $\mathbb{S}^n \times \mathbb{R}$
\begin{eqnarray*}
H_{x}(r)=\{X \in \mathbb{H}^{n+1}: \langle X, (x, 1)\rangle=-e^r\},
\end{eqnarray*}
where $x \in \mathbb{S}^n$ and $r \in \mathbb{R}$ represents the signed geodesic distance from the ``north pole" $N =(0, 1)\in
\mathbb{H}^{n+1}$.
The interior of the horosphere is called the horo-ball and we denote by
\begin{eqnarray*}
B_{x}(r)=\{X \in \mathbb{H}^{n+1}: 0>\langle X, (x,
1)\rangle>-e^r\}.
\end{eqnarray*}
If we use the Poincar\'e ball model $\mathbb{B}^{n+1}$ of $\mathbb{H}^{n+1}$, then $B_x(r)$ corresponds to
an $(n+1)$-dimensional ball which tangents to $\partial \mathbb{B}^{n+1}$ at $x$. Furthermore, $B_x(r)$ contains the origin for $r>0$.
\begin{definition}
A compact domain $\Omega\subset \mathbb{H}^{n+1}$ (or its boundary $\partial \Omega$) is horospherically convex
(or $h$-convex for short) if every boundary point $p$ of
$\partial \Omega$ has a supporting horo-ball, i.e. a horo-ball $B$
such that $\Omega\subset \overline{B}$ and $p \in \partial B$. When $\Omega$ is smooth, it is $h$-convex if
and only if the principal curvatures of $\partial \Omega$ are greater
than or equal to $1$.

For a smooth compact domain $\Omega$, we say
$\Omega$ (or $\partial \Omega$) is uniformly $h$-convex if its principal curvatures are greater than $1$.
\end{definition}

\begin{definition}\label{HG}
Let $\Omega\subset \mathbb{H}^{n+1}$ be a $h$-convex compact domain. For each $X \in \partial \Omega$,
$\partial \Omega$ has a supporting horo-ball $B_{x}(r)$
for some $r \in \mathbb{R}$ and $x \in \mathbb{S}^n$. Then the horospherical
Gauss map $G: \partial \Omega\rightarrow S_{\infty}=(\mathbb{S}^n, g_{\infty})$ of $\Omega$ (or $\partial \Omega$)
is defined by
\begin{eqnarray*}
G(X)=x, \quad g_{\infty}(X)=e^{2r}\sigma,
\end{eqnarray*}
where $\sigma$ is the canonical $\mathbb{S}^n$ metric.
\end{definition}

Note that the canonical $\mathbb{S}^n$ metric $\sigma$ is used in order to measure geometric quantities associated to the
Euclidean Gauss map of a hypersurface in $\mathbb{R}^{n+1}$. However, Espinar-G\'alvez-Mirain
explained in detail why we use the horospherical metric $g_{\infty}$ on $\mathbb{S}^n$
for measuring geometrical quantities with respect to the horospherical Gauss map (see section 2 in \cite{Esp09}).
Let $M$ be a $h$-convex hypersurface in $\mathbb{H}^{n+1}$.
For each $X \in M$, $M$  has a supporting horo-ball $B_{x}(r)$.
Then, we have (see (5.3) in \cite{And20} or (2.2) in \cite{Li-Xu})
\begin{eqnarray*}
X-\nu=e^{-r}\langle x, 1\rangle,
\end{eqnarray*}
where $\nu$ is the unit outward vector of $M$. Differentiating the above equation gives
\begin{eqnarray}\label{dG}
\langle d G, d G\rangle_{g_{\infty}}=(\mathcal{W}-I)^2,
\end{eqnarray}
where $\mathcal{W}$ is Weingarten matrix of $M$ and $I$ is the identity matrix.

The relation \eqref{dG} declares that the matrix $\mathcal{W}-I$ plays the role in $h$-convex hyperbolic geometry
of the matrix $\mathcal{W}$ in convex Euclidean geometry. This fact
motivates Andrews-Chen-Wei \cite{And20} to define
the shifted Weingarten matrix by
$\widetilde{\mathcal{W}}:=\mathcal{W}-I$.
Clearly, a smooth hypersurface $M\subset \mathbb{H}^{n+1}$ is uniformly $h$-convex
if and only if its shifted Weingarten matrix $\widetilde{\mathcal{W}}$ is positive definite. Thus, $G$ is a diffeomorphism from $M$ to $\mathbb{S}^n$ if $M$ is uniformly $h$-convex.

Let $\kappa_1, . . ., \kappa_n$ be principal curvatures of $M$, the shifted principal curvatures is defined by \cite{And20}
\begin{eqnarray*}
(\tilde{\kappa}_1, . . . , \tilde{\kappa}_n):=(\kappa_1 -1, . . ., \kappa_n -1),
\end{eqnarray*}
which are eigenvalues of the shifted Weingarten matrix
$\widetilde{\mathcal{W}}$.  Thus, the hyperbolic curvature radii take the form (see Definition 8 in \cite{Esp09})
\begin{eqnarray*}
\mathcal{R}_i:=\frac{1}{\kappa_i-1}.
\end{eqnarray*}
Espinar-G\'alvez-Mirain showed the hyperbolic curvature radii plays the role in hyperbolic space of the Euclidean
curvature radii from several different perspectives and used it to extend the
Christoffel problem \cite{Chr65, Fi} to hyperbolic space.

Using the shifted principal curvatures, the $k$-th shifted mean curvature for $1\leq k\leq n$ can be defined by
\begin{eqnarray*}
\widetilde{H}_k=\sum_{i_1<i_2< \cdot\cdot\cdot<
i_k}\tilde{\kappa}_{i_1}\cdot\cdot\cdot\tilde{\kappa}_{i_k},
\end{eqnarray*}
which is used by Andrews-Chen-Wei \cite{And20} to define
the modified quermassintegrals and the corresponding Alexandrov-Fenchel type inequalities
have proved in \cite{And20, Hu20} by using shifted curvature flows in hyperbolic space.
Later, Wang-Wei-Zhou \cite{Wang20} studied
inverse shifted curvature flow in hyperbolic space.

In this paper, we consider prescribed shifted Weingarten curvatures problem in hyperbolic space
which is motivated by the work of Guan-Guan \cite{Guan-g} on
the similar problem \eqref{GG} in Euclidean space .

\begin{problem}\label{LX}
Let $1\leq k\leq n$ be a fixed integer. For a given smooth positive function $\tilde{f}(x)$ defined on $\mathbb{S}^n$,
does there exist a smooth, closed and uniformly $h$-convex hypersurface $M\subset \mathbb{H}^{n+1}$ satisfying
\begin{eqnarray}\label{PC}
\widetilde{H}_{k}(G^{-1}(x))=\tilde{f}(x).
\end{eqnarray}
\end{problem}

This problem is also a special case of the generalized Christoffel problem (see (5.17) in \cite{Esp09}).
For $k=n$, this problem is prescribed shifted Gauss problem which has been studied in \cite{Li-Xu, Chen23}.
In this paper, we solve Problem \ref{LX} when $\tilde{f}$ is an even function, i.e.
$\tilde{f}(x)=\tilde{f}(-x)$ for all $x \in \mathbb{S}^n$.
\begin{theorem}\label{Main}
Assume $1\leq k<n$, there exists a smooth, closed, origin-symmetric and uniformly $h$-convex hypersurface in
$\mathbb{H}^{n+1}$
satisfying the equation \eqref{PC} for any smooth positive even function $\tilde{f}$ defined on $\mathbb{S}^n$.
\end{theorem}

In Sect. 2, we will show that Problem \ref{LX} is reduced to solve a Hessian quotient
equation on $\mathbb{S}^n$.  After establishing the a priori estimates for solutions to the Hessian quotient
equation in Sect. 3, we will use the degree theory to prove Theorem \ref{Main} in Sect. 4.

\section{The Hessian quotient
equation associated to Problem \ref{LX}}

In this section, using the the horospherical support function of a
$h$-convex hypersurface in $\mathbb{H}^{n+1}$, we can reduce Problem \ref{LX} to solve a Hessian quotient
equation on $\mathbb{S}^n$.

We will continue to review $h$-convex geometry in hyperbolic space in Section 1
followed by Section 5 in \cite{And20} and Section 2 in \cite{Li-Xu}.
We also work
in the hyperboloid model of $\mathbb{H}^{n+1}$ as Section 1.
Let $\Omega$ be a $h$-convex compact domain in $\mathbb{H}^{n+1}$. Then
for each $x\in \mathbb{S}^n$ we define the horospherical support function
of $\Omega$ (or $\partial \Omega$) in direction $x$ by
\begin{eqnarray*}
u(x):=\inf\{s \in \mathbb{R}: \Omega\subset \overline{B}_{x}(s)\}.
\end{eqnarray*}
We also have the alternative characterisation
\begin{eqnarray}\label{SD}
u(x)=\sup\{\log(-\langle X, (x, 1)\rangle): X \in \Omega\}.
\end{eqnarray}
The support function completely determines a
$h$-convex compact domain $\Omega$, as an intersection of horo-balls:
\begin{eqnarray*}
\Omega=\bigcap_{x \in \mathbb{S}^n}\overline{B}_{x}(u(x)).
\end{eqnarray*}

Since $G$ is a diffeomorphism from $\partial \Omega$ to $\mathbb{S}^n$ for a compact uniformly $h$-convex domain $\Omega$. Then,
$\overline{X}=G^{-1}$ is a smooth embedding from $\mathbb{S}^n$ to $\partial \Omega$ and
$\overline{X}$ can be written in terms of the support function $u$, as follows:
\begin{eqnarray}\label{X}
\overline{X}(x)=\frac{1}{2}\varphi(-x, 1)+\frac{1}{2}
\Big(\frac{|D \varphi|^2}{\varphi}+\frac{1}{\varphi}\Big)(x, 1)-(D\varphi, 0),
\end{eqnarray}
where $\varphi=e^u$ and
$D$ is the Levi-Civita connection of the standard metric $\sigma$ of $\mathbb{S}^{n}$. Then,
after choosing normal coordinates around $x$ on $\mathbb{S}^{n+1}$, we
express the shifted Weingarten matrix in the horospherical support function
(see (1.16) in \cite{And20}, Lemma 2.2 in \cite{Li-Xu})
\begin{eqnarray*}
\widetilde{\mathcal{W}}=\Big(\varphi A[\varphi]\Big)^{-1},
\end{eqnarray*}
where
\begin{eqnarray*}
A[\varphi]=D^2\varphi-\frac{1}{2}\frac{|D\varphi|^2}{\varphi}I+\frac{1}{2}\Big(\varphi-\frac{1}{\varphi}\Big)I
\end{eqnarray*}
Thus, $\Omega \subset \mathbb{H}^{n+1}$ is  uniformly $h$-convex if and only if
the matrix $A[\varphi]$ is positive definite.

So, Problem \ref{LX} is equivalent to find a smooth positive solution $\varphi(x)$ with $A[\varphi(x)]>0$ for all $x \in \mathbb{S}^n$
to the equation
\begin{eqnarray}\label{MA}
\frac{\sigma_n(A[\varphi])}{\sigma_{n-k}(A[\varphi])}=\varphi^{-k}f(x),
\end{eqnarray}
where $f=\tilde{f}^{-1}$ and $\sigma_k(A)$ is the $k$-th elementary symmetric function of a symmetric matrix $A$. Next, we will give the definition of the elementary symmetric functions and review their basic properties which could be found in
\cite{L96}.

\begin{definition}
For any $k=1,2,\cdots,n$, we set
$$\sigma_{k}(\lambda)=\sum\limits_{1\le i_1<i_2<\cdots < i_k\le n}
\lambda_{{i}_{1}}\lambda_{{i}_{2}}\cdots\lambda_{{i}_{k}},\leqno(1.1.1)$$
for any $\lambda=(\lambda_{1},\cdots,\lambda_{n})\in\mathbb{R}^{n}$ and set $\sigma_0(\lambda)=1$.
Let $\lambda_1(A)$, ..., $\lambda_n(A)$ be the eigenvalues of $A$ and
denote by $\lambda(A)=(\lambda_1(A), ..., \lambda_n(A)$. We define by $\sigma_{k}(A)=\sigma_{k}(\lambda(A))$.
\end{definition}

We recall that the Garding's cone is defined as
$$\Gamma_{k}=\{\lambda\in \mathbb{R}^{n}:\sigma_{i}(\lambda)>0,\forall 1\le i\le k\}.$$

The following two propositions will be used later.
\begin{proposition}(Generalized Newtown-MacLaurin
inequality) For $\lambda\in\Gamma_{k}$ and $k>l\geq0$, $r>s\geq0$,
$k\geq{r}$, $l\geq{s}$, we have
\begin{eqnarray}\label{GNM}
\bigg[\frac{\sigma_{k}(\lambda)/C_{n}^{k}}{\sigma_{l}(\lambda)/C_{n}^{l}}\bigg]^{\frac{1}{k-l}}
\leq\bigg[\frac{\sigma_{r}(\lambda)/C_{n}^{r}}{\sigma_{s}(\lambda)/
C_{n}^{s}}\bigg]^{\frac{1}{r-s}}
\end{eqnarray}
and the equality holds if $\lambda_{1}=...=\lambda_{n}>0.$
\end{proposition}

\begin{proposition}\label{con}
For $\lambda\in\Gamma_{k}$, we have for $n\geq k>l\geq0$
\begin{eqnarray}\label{el}
\frac{\partial}{\partial \lambda_i}\bigg[\frac{\sigma_{k}(\lambda)}{\sigma_{l}(\lambda)}\bigg]>0
\end{eqnarray}
for $1\leq i\leq n$
and
\begin{eqnarray*}
\bigg[\frac{\sigma_{k}(\lambda)}{\sigma_{l}(\lambda)}\bigg]^{\frac{1}{k-l}}
\end{eqnarray*}
is a concave function.
Moreover, we have
\begin{eqnarray}\label{sum}
\sum_{i=1}^{n}\frac{\partial[\frac{\sigma_{k}}{\sigma_{l}}]
^{\frac{1}{k-l}}(\lambda)}{\partial\lambda_{i}}
\geq[\frac{C_{n}^{k}}{C_{n}^{l}}]^{\frac{1}{k-l}}.
\end{eqnarray}
\end{proposition}

\section{The a priori estimates}

For convenience, in the following of this paper, we always assume that $f$ is
a smooth positive, even function on $\mathbb{S}^n$ and $\varphi$ is a smooth even solution
to the equation \eqref{MA} with $A[\varphi]>0$.
Moreover, let $M$ be the smooth, closed and uniformly $h$-convex hypersurface in $\mathbb{H}^{n+1}$ with the horospherical support function $u=\log \varphi$.
Clearly, $M$ is symmetry with the origin and $\varphi(x)> 1$ for $x \in \mathbb{S}^n$.

The following easy and important equality is key for
the $C^0$ estimate.
\begin{lemma}
We have
\begin{eqnarray}\label{c0-12}
\frac{1}{2}\Big(\max_{\mathbb{S}^n}\varphi+\frac{1}{\max_{\mathbb{S}^n}\varphi}\Big)\leq \min_{\mathbb{S}^n}\varphi.
\end{eqnarray}
\end{lemma}

\begin{proof}
The inequality can be found in the proof of Lemma 7.2 in \cite{Li-Xu}. For completeness, we give a proof here.
Assume that $\varphi(x_1)=\max_{\mathbb{S}^n}\varphi$ and denote $\overline{X}(x_1)=G^{-1}(x_1)$ as before. Then, we have for any $x \in \mathbb{S}^n$ by
the definition of the horospherical support function \eqref{SD}
\begin{eqnarray*}
-\langle \overline{X}(x_1), (x, 1)\rangle\leq \varphi(x), \quad \forall x \in \mathbb{S}^n.
\end{eqnarray*}
Substituting the expression \eqref{X} for $\overline{X}$ into the above equality yields
\begin{eqnarray}\label{D0-3}
\frac{1}{2}\varphi(x_1)(1+\langle x_1, x\rangle)+\frac{1}{2}\frac{1}{\varphi(x_1)}(1-\langle x_1, x\rangle)\leq \varphi(x),
\end{eqnarray}
where we used the fact $D\varphi(x_1)=0$. Note that $\varphi(x_1)\geq 1$, we find from \eqref{D0-3}
\begin{eqnarray}\label{D0-4}
\frac{1}{2}\Big(\varphi(x_1)+\frac{1}{\varphi(x_1)}\Big)\leq \varphi(x) \quad \mbox{for} \quad \langle x, x_1\rangle\geq 0.
\end{eqnarray}
Since $\varphi$ is even, we can assume
that the minimum point $x_0$ of $\varphi(x)$ satisfies $\langle x_0, x_1\rangle\geq 0$. Thus,
the equality \eqref{c0-12} follows that from \eqref{D0-4}.
\end{proof}

Now, we use the maximum principle to get the $C^0$-estimate.
\begin{lemma}\label{C-C0}
We have
\begin{eqnarray}\label{C0}
0<\frac{1}{C}\leq u(x)\leq C, \quad \forall \ x \in \mathbb{S}^n,
\end{eqnarray}
where $C$ is a positive constant depending on $k$, $n$ and $f$.
\end{lemma}

\begin{proof}
Applying the maximum principle, we have from the equation \eqref{MA}
\begin{eqnarray*}
(\max_{\mathbb{S}^n}\varphi)^{2k}\frac{1}{2^k}\bigg[1-\frac{1}{(\max_{\mathbb{S}^n}\varphi)^2}\bigg]^k\geq C>0
\end{eqnarray*}
and
\begin{eqnarray*}
(\min_{\mathbb{S}^n}\varphi)^{2k}\frac{1}{2^k}\bigg[1-\frac{1}{(\min_{\mathbb{S}^n}\varphi)^2}\bigg]^k\leq C.
\end{eqnarray*}
Since the function
\begin{eqnarray*}
g(x)=x^{2k}\frac{1}{2^k}\bigg(1-\frac{1}{x^2}\bigg)^k
\end{eqnarray*}
is increasing in $[1, +\infty)$, $g(1)=0$ and $g(+\infty)=+\infty$, we obtain
\begin{eqnarray}\label{c0-11}
\min_{\mathbb{S}^n}\varphi\leq C, \quad \mbox{and} \quad \max_{\mathbb{S}^n}\varphi\geq C>1.
\end{eqnarray}
Combining \eqref{c0-11} and \eqref{c0-12}, we find
\begin{eqnarray*}
1<C\leq\min_{\mathbb{S}^n}\varphi\leq \max_{\mathbb{S}^n}\varphi\leq C^{\prime},
\end{eqnarray*}
which implies that
\begin{eqnarray*}
0<\frac{1}{C}\leq\min_{\mathbb{S}^n}u\leq \max_{\mathbb{S}^n}u\leq C.
\end{eqnarray*}
So, we complete the proof.
\end{proof}

As a corollary, we have the gradient estimate from Lemma 7.3 in \cite{Li-Xu}.

\begin{corollary}\label{C-C1}
We have
\begin{eqnarray}\label{C1}
|D\varphi(x)|\leq C, \quad \forall \ x \in \mathbb{S}^n,
\end{eqnarray}
where $C$ is a positive constant depending only on the constant in Lemma \ref{C-C0}.
\end{corollary}

We give some notations before considering the $C^2$ estimate.
Denote by
\begin{eqnarray*}
U_{ij}=\varphi_{ij}-\frac{1}{2}\frac{|D\varphi|^2}{\varphi}\delta_{ij}+\frac{1}{2}(\varphi-\frac{1}{\varphi})\delta_{ij}
\end{eqnarray*}
and
\begin{eqnarray*}
F(U)=\Big[\frac{\sigma_n(U)}{\sigma_{n-k}(U)}\Big]^{\frac{1}{k}}, \quad F^{ij}=\frac{\partial F}{\partial U_{ij}},
\quad F^{ij, st}=\frac{\partial^2 F}{\partial U_{ij}\partial U_{st}}.
\end{eqnarray*}

\begin{lemma}
We have for $1\leq i\leq n$
\begin{eqnarray}\label{C2}
\lambda_i(U(x))\leq C, \quad \forall \ x \in \mathbb{S}^n,
\end{eqnarray}
where $\lambda_1(U), ..., \lambda_n(U)$ are eigenvalues of the matrix $U$ and $C$ is a positive constant depending
only on the constant in Lemma \ref{C-C0} and Corollary \ref{C-C1}.
\end{lemma}

\begin{proof}
Since
\begin{eqnarray*}
\lambda_i(U)\leq \mathrm{tr} U=\Delta \varphi-\frac{n}{2\varphi}|D\varphi|^2+\frac{n}{2}(\varphi-\frac{1}{\varphi}), \quad \forall 1\leq i\leq n,
\end{eqnarray*}
it is sufficient to prove $\Delta\varphi\leq C$ in view of $C^0$ estimate \eqref{C0} and $C^1$ estimate \eqref{C1}.
Moreover, these two estimates \eqref{C0} and \eqref{C1} together with the positivity of the matrix $U$ imply $\lambda_{i}(D^2\varphi)\geq -C$ for $1\leq i\leq n$.
Thus, we find
\begin{eqnarray}\label{C2-3}
|\lambda_i(D^2\varphi)|\leq C \max \{ \max_{\mathbb{S}^n}\Delta \varphi, 1\}, \quad \forall 1\leq i\leq n.
\end{eqnarray}
We take the auxiliary function
$$W(x)=\Delta \varphi.$$
Assume $x_0$ is the maximum point of $W$. After an appropriate
choice of the normal frame at $x_0$, we further assume $U_{ij}$, hence $\varphi_{ij}$ and $F^{ij}$ is diagonal at the point $x_0$. Then,
\begin{equation}\label{W1}
W_i(x_0)=\sum_{q}\varphi_{qqi}=0,
\end{equation}
and
\begin{equation}\label{W2}
W_{ii}(x_0)=\sum_{q}\varphi_{qqii}\leq0.
\end{equation}
Using \eqref{W2} and the positivity of $F^{ij}$ given by \eqref{el}, we arrive at $x_0$ if $\Delta \varphi$ is large enough
\begin{eqnarray*}
0&\ge&\sum_iF^{ii}W_{ii}=\sum_iF^{ii}\sum_{q}\varphi_{qqii}\geq\sum_iF^{ii}\sum_{q}\big(\varphi_{iiqq}-C\Delta \varphi\big),
\end{eqnarray*}
where we use Ricci identity and the equality \eqref{C2-3} to get the last inequality. Thus it
follows from the definition of $U$, \eqref{C0}, \eqref{C1} and \eqref{W1},
\begin{eqnarray}\label{C2-1}
0&\ge&\sum_iF^{ii}\sum_{q}\bigg[U_{iiqq}+\Big(\frac{1}{2\varphi}|D\varphi|^2\Big)_{qq}-
\frac{1}{2}\Big(\varphi-\frac{1}{\varphi}\Big)_{qq}\bigg]
-C\sum_{i}F^{ii}\Delta \varphi\nonumber\\
&\ge&\sum_iF^{ii}\sum_{q}\bigg[U_{iiqq}+\frac{1}{\varphi}(\varphi_{qq})^2-C\Delta\varphi-C\bigg]
-C\sum_iF^{ii}\Delta \varphi.
\end{eqnarray}
Differentiating the equation \eqref{MA} twice gives
\begin{eqnarray*}
F^{ii}U_{iiqq}+F^{ij, st}U_{ijq}U_{stq}=(\varphi^kf)_{qq}.
\end{eqnarray*}
Since $F$ is concave (see Proposition \ref{con}), it yields
\begin{eqnarray}\label{C2-2}
F^{ii}U_{iiqq}\geq -C\Delta \varphi-C,
\end{eqnarray}
where we used \eqref{C0} and \eqref{C1}.
Substituting \eqref{C2-2} into \eqref{C2-1} and using $(\Delta \varphi)^2 \leq n \sum_{q}(\varphi_{qq})^2$, we have
\begin{eqnarray}\label{Q}
0\ge\bigg(C(\Delta \varphi)^2-C\Delta \varphi-C\bigg)\sum_{i}F^{ii}-C\Delta \varphi-C.
\end{eqnarray}
Applying the inequality \eqref{sum}, we find
\begin{eqnarray}\label{C2-8}
\sum_{i}F^{ii}\geq C.
\end{eqnarray}
Then we conclude at $x_0$ by combining the inequalities \eqref{C2-8} and \eqref{Q}
\begin{eqnarray*}
C\geq |\Delta \varphi|^2
\end{eqnarray*}
if $\Delta \varphi$ is chosen large enough.
So, we complete the proof.
\end{proof}

\section{The proof of the main theorem}

In this section, we use the degree theory for nonlinear elliptic
equations developed in \cite{Li89} to prove Theorem \ref{Main}. Such approach was also used
in prescribed curvature problem for star-shaped hypersurfaces \cite{An, Li02, Jin, Li-Sh}, prescribed curvature problem for
convex hypersurfaces
\cite{Guan02, Guan21} and the Gaussian
Minkowski type problem \cite{Huang-Xi21, Liu22, Feng1, Feng2}.

For the use of the degree theory, the uniqueness of constant solutions to the equation
\eqref{MA} is important for us.

\begin{lemma}\label{U-C}
The $h$-convex solutions to the equation
\begin{eqnarray}\label{MA-c}
\varphi^k\frac{\sigma_n}{\sigma_{n-k}}(A[\varphi(x)])=\gamma
\end{eqnarray}
with $\varphi>1$ are given by
\begin{eqnarray*}
\varphi(x)=\Big(1+2\gamma^{\frac{1}{k}}\Big)^{\frac{1}{2}}\Big(\sqrt{|x_0|^2+1}-\langle x_0, x\rangle\Big),
\end{eqnarray*}
where $x_0 \in \mathbb{R}^{n+1}$. In particular, $\varphi(x)=\Big(1+2\gamma^{\frac{1}{k}}\Big)^{\frac{1}{2}}$ is the unique even solution.
\end{lemma}

\begin{proof}
This lemma is a corollary of Proposition 8.1 in \cite{Li-Xu}, its proof is similar to that of Theorem 8.1 (7).
\end{proof}

Now, we begin to use the degree theory to prove Theorem \ref{Main}.
After establishing the  a priori estimates \eqref{C0}, \eqref{C1} and \eqref{C2} and noting
that $\sigma_n(U)\geq C>0$ which is given by using generalized Newtown-MacLaurin
inequality \eqref{GNM} and the equation \eqref{MA}, we know that the
equation \eqref{MA} is uniformly elliptic, i.e.
\begin{eqnarray}\label{C2+++}
\lambda_i(U(x))\geq C>0, \quad \forall \ x \in \mathbb{S}^n,
\end{eqnarray}
where $\lambda_1(U), ..., \lambda_n(U)$ are eigenvalues of the matrix $U$.
From Evans-Krylov estimates \cite{Eva82, Kry83} and Schauder estimates \cite{GT}, we have
\begin{eqnarray}\label{C2+}
|\varphi|_{C^{4,\alpha}(\mathbb{S}^n)}\leq C
\end{eqnarray}
for any smooth, even and uniformly $h$-convex solution $\varphi$ to the equation \eqref{MA}.
We define
\begin{eqnarray*}
\mathcal{B}^{2,\alpha}(\mathbb{S}^n)=\{\varphi \in
C^{2,\alpha}(\mathbb{S}^n): \varphi \ \mbox{is even}\}
\end{eqnarray*}
and
\begin{eqnarray*}
\mathcal{B}_{0}^{4,\alpha}(\mathbb{S}^n)=\{\varphi \in
C^{4,\alpha}(\mathbb{S}^n): A[\varphi]>0 \ \mbox{and} \ \varphi \ \mbox{is even}\}.
\end{eqnarray*}
Let us consider $$\mathcal{L}(\cdot, t): \mathcal{B}_{0}^{4,\alpha}(\mathbb{S}^n)\rightarrow
\mathcal{B}^{2,\alpha}(\mathbb{S}^n),$$ which is defined by
\begin{eqnarray*}
\mathcal{L}(\varphi, t)=\frac{\sigma_n}{\sigma_{n-k}}(U)-\varphi^{-k} [(1-t)\gamma+t f],
\end{eqnarray*}
where the constant $\gamma$ will be chosen later and $U$ is denoted as before
\begin{eqnarray*}
U=D^2\varphi-\frac{1}{2}\frac{|D\varphi|^2}{\varphi}I+\frac{1}{2}\Big(\varphi-\frac{1}{\varphi}\Big)I.
\end{eqnarray*}

Let $$\mathcal{O}_R=\{\varphi \in \mathcal{B}_{0}^{4,\alpha}(\mathbb{S}^n):
1+\frac{1}{R}< \varphi, \ \frac{1}{R} I< U, \ |\varphi|_{C^{4,\alpha}(\mathbb{S}^n)}<R\},$$ which clearly is an open
set of $\mathcal{B}_{0}^{4,\alpha}(\mathbb{S}^n)$. Moreover, if $R$ is
sufficiently large, $\mathcal{L}(\varphi, t)=0$ has no solution on $\partial
\mathcal{O}_R$ by the a priori estimates established in \eqref{C0}, \eqref{C2+++} and \eqref{C2+}.
Therefore the degree $\deg(\mathcal{L}(\cdot, t), \mathcal{O}_R, 0)$ is
well-defined for $0\leq t\leq 1$. Using the homotopic invariance of
the degree (Proposition 2.2 in \cite{Li89}), we have
\begin{eqnarray}\label{hot}
\deg(\mathcal{L}(\cdot, 1), \mathcal{O}_R, 0)=\deg(\mathcal{L}(\cdot, 0), \mathcal{O}_R, 0).
\end{eqnarray}

Lemma \ref{U-C} tells us that
$\varphi=c$ is the unique even solution for $\mathcal{L}(\varphi, 0)=0$ in $\mathcal{O}_R$.
Direct calculation show that the linearized operator of $\mathcal{L}$ at
$\varphi=c$ is
\begin{eqnarray*}
L_{c}(\psi)=a(n, k, c)(\Delta_{\mathbb{S}^n}+n)\psi,
\end{eqnarray*}
where
\begin{eqnarray*}
a(n, k, c)=\frac{k}{n}\frac{2}{c-c^{-1}}\frac{\sigma_n(\frac{1}{2}(c-c^{-1})I)}{\sigma_{n-k}(\frac{1}{2}(c-c^{-1})I)}.
\end{eqnarray*}
Since
$\Delta_{\mathbb{S}^n}\psi+n\psi=0$
has the unique even solution $\psi=0$,
$L_{c}$ is an invertible operator. So, we have by Proposition
2.3 in \cite{Li89}
\begin{eqnarray*}
\deg(\mathcal{L}(\cdot, 0), \mathcal{O}_R, 0)=\deg(L_{c_0}, \mathcal{O}_R, 0).
\end{eqnarray*}
Because
the eigenvalues of the Beltrami-Laplace operator $\Delta$ on $\mathbb{S}^n$ are strictly less than
$-n$ except for the first two eigenvalues $0$ and $-n$,
there is only one positive eigenvalue $na(n, k, c)$ of $L_{c}$
with multiplicity $1$.
Then we have by Proposition
2.4 in \cite{Li89}
\begin{eqnarray*}
\deg(\mathcal{L}(\cdot, 0), \mathcal{O}_R, 0)=\deg(L_{c_0}, \mathcal{O}_R, 0)=-
1.
\end{eqnarray*}
Therefore, it follows from \eqref{hot}
\begin{eqnarray*}
\deg(\mathcal{L}(\cdot, 1), \mathcal{O}_R; 0)=\deg(\mathcal{L}(\cdot, 0), \mathcal{O}_R, 0)=-
1.
\end{eqnarray*}
So, we obtain a solution at $t=1$. This completes the proof of
Theorem \ref{Main}.

\end{document}